\documentclass[svgraybo,hyperref,optionalrh]{svmult}
\usepackage{amsfonts, amsmath, amssymb}
\usepackage{mathtools}
\usepackage{mathrsfs}
\usepackage{mathptmx}
\usepackage{booktabs}
\usepackage{protosem}
\usepackage{caption}
\usepackage{epigraph}
\usepackage{enumitem}
\usepackage{subcaption}
\usepackage{tikz}
\usepackage{float}
\usepackage[numbers,sort&compress]{natbib}

\usepackage{todonotes}

\usepackage{hyperref}

\newcommand{\R}{\mathbb{R}}

\newcommand{\N}{\mathbb{N}}

\newcommand{\dist}{\mathrm{dist}}

\newcommand{\Cheb}{\mathrm{Cheb}}
\newcommand{\Oc}{\mathcal{O}}

\DeclareMathOperator*{\argmin}{arg\,min}

\newtheorem{experiment}{\sc Experiment}
{\bf}{\it}

\usepackage{type1cm}        
\usepackage{makeidx}         
\usepackage{graphicx}        
\usepackage{multicol}        
\usepackage[bottom]{footmisc}

\usepackage{newtxtext} 
\usepackage[varvw]{newtxmath}       
\makeindex             

\begin{document}

\title*{High-order numerical integration on regular embedded surfaces}

\author{Gentian Zavalani and Michael Hecht}
\institute{Gentian Zavalani \at Center for Advanced Systems Understanding (CASUS), G\"{o}rlitz, Germany \at Technical University  Dresden, Faculty of Mathematics, Dresden, Germany \email{g.zavalani@hzdr.de}
\and Michael Hecht \at Center for Advanced Systems Understanding (CASUS), G\"{o}rlitz, Germany \at University of Wroclaw, Mathematical Institute,  Wroclaw, Poland \email{m.hecht@hzdr.de}}

\maketitle
\vspace{-8em} 
\abstract{
We present a high-order surface quadrature (HOSQ) for accurately approximating regular surface integrals on closed surfaces. The initial step of our approach rests on exploiting 
\emph{square-squeezing}--a homeomorphic bilinear square-simplex transformation, re-parametrizing any surface triangulation to a quadrilateral mesh. For each resulting quadrilateral domain we interpolate the geometry
by tensor polynomials in  Chebyshev--Lobatto grids. Posterior the tensor-product Clenshaw-Curtis quadrature is applied to compute the resulting integral. We  demonstrate efficiency, fast runtime performance, high-order accuracy, and robustness for complex geometries. 
}
\begin{keywords}
high-order integration, spectral differentiation, numerical quadrature, quadrilateral mesh
\end{keywords}

\section{Introduction}
Efficient numerical integration of surface integrals is an important ingredient in applications ranging from \emph{geometric processing}~\cite{Lachaud16}, \emph{surface--interface and colloidal sciences} \cite{zhou2005surface}, optimization of production processes \cite{Gregg_1967,riviere2009} to surface finite element methods, solving partial differential equations on curved surfaces. Given a conforming triangulation of a $2$-dimensional $C^{r+1}$-surface $S$, $r \geq 0$, a finite family
of maps $\varrho_i$ and corresponding sets $V_i \subset S$, $i=1,\dots,K$
such that
\begin{equation*}
    \varrho_i : \triangle_2 \to V_i \subseteq S\,, \quad \bigcup_{i=1}^K
    \overline{V_i} = S\,,
    \quad \bigcap_{i\neq j}  V_i \cap V_j = \emptyset\,,
\end{equation*}
and the restrictions of the $\varrho_i$ to the interior of the standard simplex
$\mathring \triangle_2$ are diffeomorphisms. 
For immersed manifolds $S \subseteq \R^3$ we will write
$D\varrho_i(\mathrm{x}) : \R^2 \to \R^3$ for the
Jacobian of the parametrization $\varrho_i$ at $\mathrm{x}$. The surface integral
of an integrable function
$f : S \to \R$ with $f \in C^r$, for each $V_i\in S$ is
\begin{equation}\label{eq:SI}
    \int_S f \,dS = \sum_{i=1}^K\int_{\triangle_2} f(\varrho_i(\mathrm{x})) g_i(\mathrm{x})\,d\mathrm{x}\,,
\end{equation}
here, the volume element is expressed as $g_i(\mathrm{x})=\sqrt{\det((D\varrho_i)^T(\mathrm{x}) D\varrho_i(\mathrm{x}))}$.

Triangulations of an embedded
manifold $S$ are frequently given as a set of flat triangles
in the embedding space $\R^3$, together with local
projections from these simplices onto $S$.  More formally, let
\begin{equation}\label{eq:triang}
 T_i \subseteq \R^3\,,
 \qquad
 i = 1,\dots,K
\end{equation}
be a set of flat triangles.
For each triangle $T_i$, we assume that there is a well-defined $C^{r+1}$-embedding $\pi_i : T_i \to S$ and an invertible affine transformation
$\tau_i : \triangle_2 \to T_i$, such that the maps $\varrho_i=\pi_i \circ \tau_i : \triangle_2 \to S$ form a conforming triangulation. Commonly, \emph{the closest-point
 projection}
\begin{equation*}
    \pi^* : \mathcal{N}_{\delta}(S) \to S\,, \quad \pi^*(\mathrm{x}) = \argmin_{\mathrm{y} \in S}\dist(\mathrm{x},\mathrm{y})
\end{equation*}
serves as a realization of the $\pi_i$.
Recall from \cite{geometricpde} that
given an open neighborhood $\mathcal{N}_{\delta}(S)=\{\mathbf{x}\in \mathbb R^{m} :  \mathrm{dist}(\mathbf{x},S)<\delta\}$ of a $C^{r+1}$-surface $S$, $r \geq 2$, with $\delta$ bounded by the reciprocal of the maximum of all principal curvatures on $S$, the closest-point projection
is well-defined on $\mathcal{N}_{\delta}(S)$ and of regularity $\pi^*\in C^{r-1}(T,S)$. 

Especially when it comes to handling complex geometries, 
in practice, simplex meshes are
typically much easier to obtain ~\cite{Persson}. Here, we extend our former approach  \cite{zavalani2023highorder}, involving the homeomorphic re-parametrization 
\begin{equation}
    \sigma_* : \square_2 \to \triangle_2 \,, \quad \triangle_2 = \Big\{\mathrm{x} \in \R^2: x_1,x_2 \geq 0 \,,\sum_{i=1}^2|x_i|\leq 1\Big\}\,, \square_2= [-1,1]^2 
\end{equation}
referred to as \emph{square-squeezing}, for re-parametrizing initial triangular meshes to quadrilateral ones. This enables pulling back interpolation and integration tasks from the standard 2-simplex $\triangle_2$ to the square $\square_2= [-1,1]^2$.

\section{Square-squeezing re-parametrization}
The main difficulty for providing a numerical approximation of \eqref{eq:SI} is obtaining the unknown derivatives $D\varrho_i$ that appear in the volume element.
One classic approach, followed also by \cite{dziuk2013finite}, is to replace the Jacobians $D\varrho_i$ by the Jacobians of a polynomial approximation, typically obtained by interpolation on a set of interpolation points in $\triangle_2$.
However, the question of how to distribute nodes in simplices in order to enable stable high-order polynomial interpolation is still not fully answered \cite{CHEN1995405,taylor2000generalized}.

To bypass this issue, we instead propose to re-parametrize the curved simplices $S_i$ over the square $\square_2=[-1,1]^2$. However, the prominent \emph{Duffy transformation} \cite{Duffy82} 
\begin{equation}\label{eq.duffy}
    \sigma_\text{Duffy}: \square_2 \to \triangle_2\,,
    \quad
    \sigma_{\text{Duffy}}(x,y) = \Big(\frac{1}{4}\left(1+x\right)\left(1-y\right),\frac{1+y}{2}\Big)\,,
\end{equation}
collapses the entire upper edge of the square to a single vertex, which renders it to be no homeomorphism and, thus, to be infeasible for the re-parametrization task.

To introduce the alternate \emph{square-squeezing}, we re-scale $\square_2$ to $\square_2^u =[0,1]^2$, by setting $x_1 = (x+1)/2$ and $x_2 = (y+1)/2$, and define
\begin{equation}\label{eq.rec-tri}
\sigma: \square_2 \to \triangle_2\,,
\quad \sigma(x_1,x_2) = \Big(x_1-\frac{x_1x_2}{2},x_2-\frac{x_1x_2}{2}\Big)^T.
\end{equation}
The inverse map $\sigma^{-1} : \triangle_2 \to \square_2$ is given by
\begin{equation}\label{eq:inv_ss}
\sigma^{-1}(u,v)
=
\begin{pmatrix}
1+\left(u-v\right)-\sqrt{\left(u-v\right)^2+4\left(1-u-v\right)} \\
1-\left(u-v\right)-\sqrt{\left(u-v\right)^2+4\left(1-u-v\right)}
\end{pmatrix}.
\end{equation}
Both $\sigma$ and $\sigma^{-1}$ are continuous, rendering square-squeezing to be a homeomorphism on the closed set $\square_2$, see Fig.~\eqref{fig:equi_simplex-square}, and even a smooth diffeomorphism in the interior, see \cite{zavalani2023highorder} for further details.

\begin{definition}[Quadrilateral re-parametrization]
\label{def:cubical_reparametrization} Let $S$ be a $C^{r+1}$-surface, $r\geq 0$, $\varrho_i=\pi_i \circ \tau_i : \triangle_2 \to S$ a conforming triangulation \eqref{eq:triang}, and  $\sigma : \square_2 \to \triangle_2$ be a 
homeomorphism whose restriction
$\sigma_{|\mathring{\square}_2} : \mathring{\square}_2 \to \mathring{\triangle}_2$ to the interior is a $C^r$-diffeomorphism.  We call
\begin{equation}\label{eq:mesh}
\varphi_i : \square_2 \to S\,,
\quad
\varphi_i = \varrho_i \circ \sigma = \pi_i \circ \tau_i \circ \sigma\,,
\quad
i=1,\dots,K
\,,
\end{equation}
an $r$-regular quadrilateral re-parametrization, illustrated in Fig.~\ref{fig:app_frame}.
\end{definition}

\begin{figure*}[!t]
  \begin{subfigure}{0.33\textwidth}
  \centering
    \includegraphics[width=1.0\linewidth]{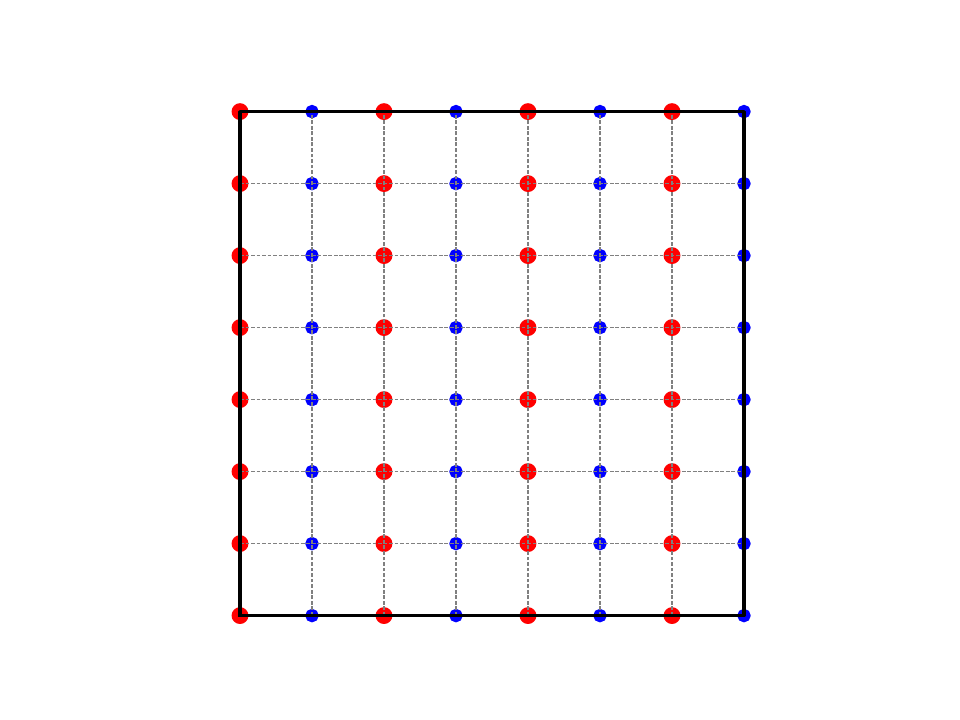}
    \caption{Standard square}
    \label{fig:f111}
  \end{subfigure}%
  \hfill
  \begin{subfigure}{0.33\textwidth}
  \centering
    \includegraphics[width=1.0\linewidth]{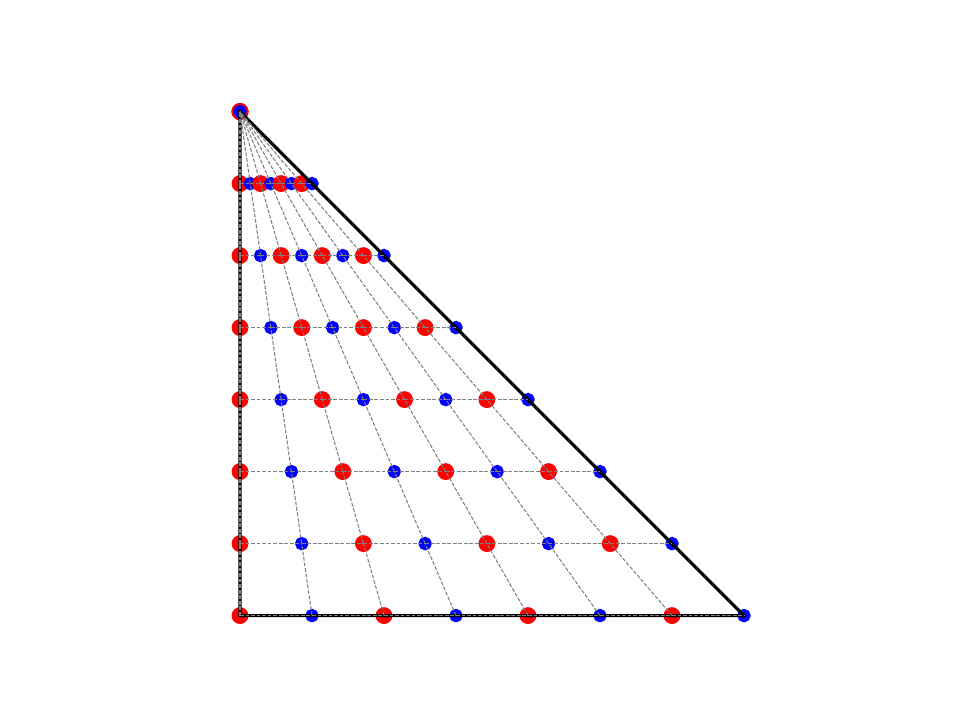}
    \caption{Duffy's transformation}
    \label{fig:f222}
  \end{subfigure}%
  \hfill
  \begin{subfigure}{0.33\textwidth}
  \centering
    \includegraphics[width=1.0\linewidth]{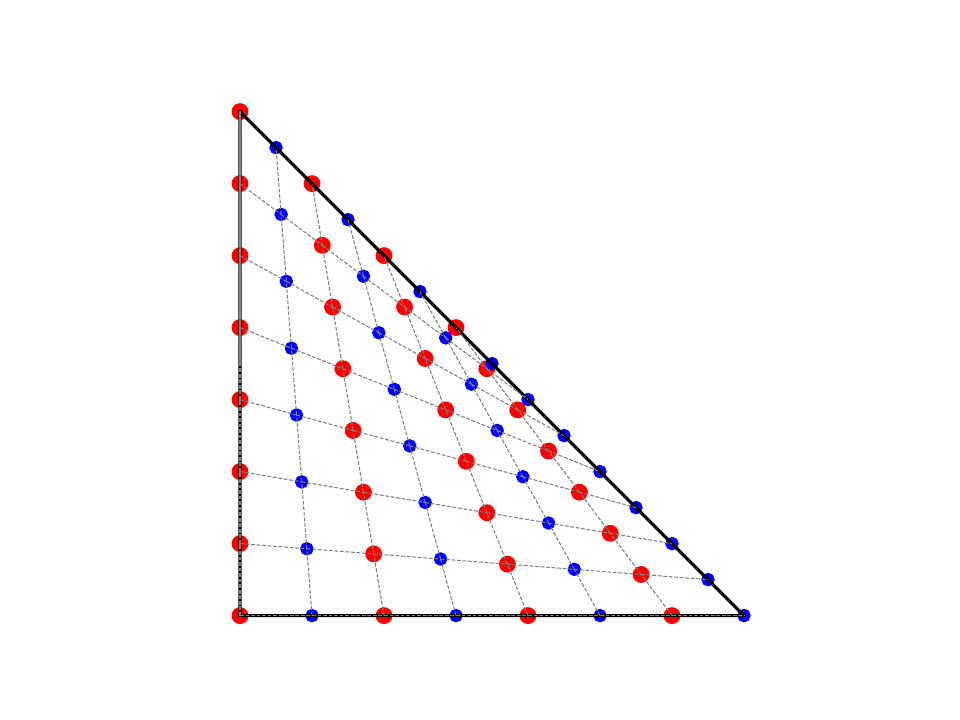}
    \caption{Square-Squeezing}
    \label{fig:f333}
  \end{subfigure}
  \caption{Bilinear square--simplex transformations:
  Deformations of equidistant grids,  under Duffy's transformation (b) and square-squeezing  (c)}
  \label{fig:equi_simplex-square}
\end{figure*}

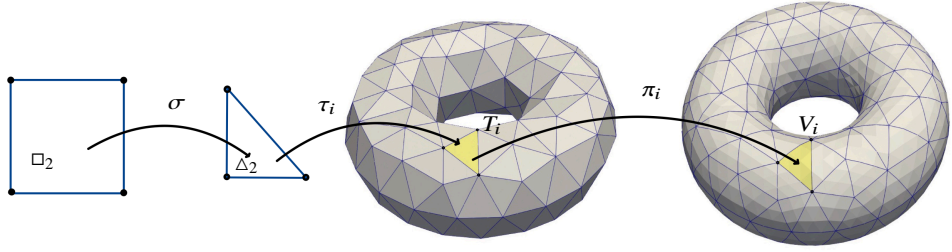
\begin{figure}[t!]
    \centering
    \begin{tikzpicture}
        \node[inner sep=0pt] at (0,0) {\includegraphics[clip,width=1.0\columnwidth]{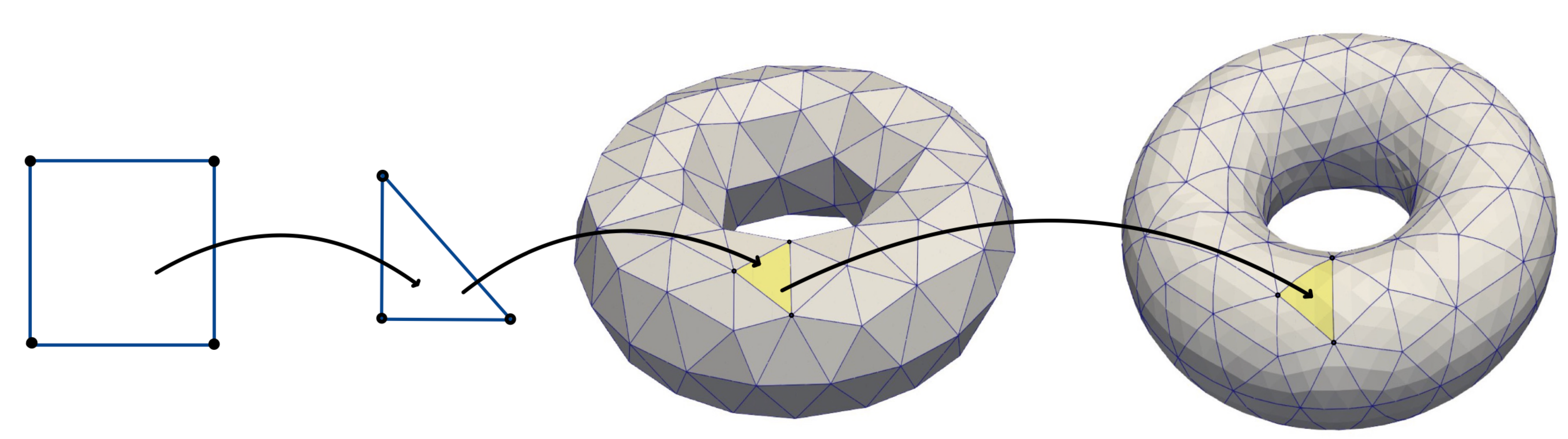}};
        
    
\node[anchor=north west] at (0.01,0.1) {
$T_i\,$};
        \node[anchor=north west] at (2.1,0.5) {$\pi_i$};

        \node[anchor=north west] at (4.2,0.1) {$V_i\,$};

        \node[anchor=north west] at (-2.2,0.3) {$\tau_i$};
        \node[anchor=north west] at (-3.3,-0.45) {$\triangle_2$};
        \node[anchor=north west] at (-4.2,0.32) {$\sigma$};
        \node[anchor=north west] at (-6.0,-0.4) {$\square_2$};
    \end{tikzpicture}
    \caption{Construction of a surface parametrization over $\triangle_2$ by
closest-point projection from a piecewise affine approximate
mesh, and re-parametrization over the square $\Box_2$.}
    \label{fig:app_frame}
\end{figure}

The  quadrilateral re-parametrization enables interpolating the 
geometry functions  $\varphi_i = \varrho_i \circ \sigma : \square_2 \to \R^3$ by 
tensor-product polynomials, as we assert next.

\section{Interpolation in the square}
Our high-order format is based on interpolation in tensor-product Chebyshev–Lobatto nodes in the square $\square_2$. We start by defining:
\begin{definition}[Lagrange polynomials \cite{MIP,minterpy}]
\label{def:LN}
Let $n\in \N$ and $G_{2,n} = \oplus_{i=1}^2 \Cheb_{n}$ be the tensorial Chebyshev–Lobatto grid, where $\Cheb_n = \left\{ \cos\left(\frac{k\pi}{n}\right) \; : \; 0 \leq k \leq n\right\}$,
indexed by a multi-index set $A_{2,n}=\{\alpha \in \N^2 : \|\alpha\|_\infty \leq n\}^2$.
For each $\alpha \in A_{2,n}$, the tensorial multivariate Lagrange polynomials are
\begin{equation}\label{eq:L}
  L_{\alpha}(x)= \prod_{i=1}^2 l_{\alpha_i,i} (x) \,, \quad l_{j,i} (x) = \prod_{k=0, k \neq j}^n \frac{x_i-p_{k,i}}{p_{j,i} - p_{k,i}}  \,.
\end{equation}
\end{definition}
The Lagrange polynomials are a basis of the polynomial space $\Pi_{2,n} =\mathrm{span}\{\mathrm{x}^\alpha = x_1^{\alpha_1}x_2^{\alpha_2}: \alpha \in A_{2,n}\}$ induced by $A_{2,n}$.
Since the $L_\alpha$ satisfy
$ L_\alpha (p_\beta) = \delta_{\alpha,\beta}$
for all $\alpha,\beta \in A_{2,n}$, $p_{\alpha} \in G_{2,n}$,
we deduce that
given a function $f :\square_2\to \R$, the interpolant $Q_{G_{2,n}}f \in \Pi_{2,n}$ of $f$ in $G_{2,n}$ is
\begin{equation}
    Q_{G_{2,n}}f = \sum_{\alpha \in A_{2,n}}f(p_\alpha)L_{\alpha}.
\end{equation}

\begin{definition}[$k^{\text{th}}$-order quadrilateral re-parametrization]
Given an  $r$-regular quadrilateral re-parametrization $\varphi_i : \square_2 \to S\,,
\quad
\varphi_i = \varrho_i \circ \sigma = \pi_i \circ \tau_i\circ \sigma\,,
\quad
i=1,\dots,K
\,,$ mesh \eqref{eq:mesh}. We say that the mesh is of order $k$ if each element has been provided as a set of nodes $\{\varphi_{i}(p_\alpha)\}_{\alpha\in A_{2,k}} $ sampled at $\{p_\alpha\}_{\alpha\in A_{2,k}}$. 
\end{definition}

This implies that on each element, we can approximate the $r$-regular quadrilateral re-parametrization maps through interpolation using the nodes $\{\varphi_{i}(p_\alpha)\}_{\alpha\in A_{2,k}}$. This involves computing a $k^{\text{th}}$-order vector-valued polynomial approximation:
\begin{equation}\label{main.poly}
    Q_{G_{2,k}}\varphi_{i}\left(\mathrm{x}\right) =\sum_{\alpha \in A_{2,k}}\varphi_{i}(p_\alpha)L_{\alpha} ,\quad i=1,\ldots, K
\end{equation}

The $r$-regular quadrilateral re-parametrization maps partial derivatives can be computed using numerical spectral differentiation \cite{trefethen2000spectral}. These derivatives are then utilized to create an approximation of the metric tensor, $g_i$, for each element.

To obtain partial derivatives, forming the Jacobian matrix $DQ_{d,k}\varphi_i$ on the reference square \(\square_2\), Kronecker products are employed. Let \(D \in \mathbb{C}^{(k+1) \times (k+1)}\) be the one-dimensional spectral differentiation matrix \cite{trefethen2000spectral} associated with Chebyshev–Lobatto nodes on the interval \([-1,1]\), and let \(I \in \mathbb{C}^{(k+1) \times (k+1)}\) be the identity matrix. The differentiation matrices in the \(x\) and \(y\) directions on the reference square \(\square_2\) are defined as follows:
\begin{equation}
    D_{\mathrm{x}}Q_{G_{2,k}}\varphi_{i} = D \otimes I, \quad D_{\mathrm{y}}Q_{G_{2,k}}\varphi_{i} = I \otimes D,
\end{equation}
where \(D_{\mathrm{x}}Q_{G_{2,k}}\varphi_{i}, D_{\mathrm{y}}Q_{G_{2,k}}\varphi_{i} \in \mathbb{C}^{(k+1)^2 \times (k+1)^2}\).

Representing the Jacobian as $DQ_{d,k}\varphi_i = [D_{\mathrm{x}}Q_{d,k}\varphi_i\,|\, D_{\mathrm{y}}Q_{d,k}\varphi_i]$, the ingredients above realise the HOSQ, computing the surface integral as:
\begin{align}
\int_S f\,dS &\approx \sum_{i=1}^K\int_{\square_2} f(\varphi(\mathrm{x}))\sqrt{\det((DQ_{d,k}\varphi_i(\mathrm{x}))^T DQ_{d,k}\varphi_i(\mathrm{x}))}\,d\mathrm{x}\label{eq:HOSQ1}\\ 
&\approx
\sum_{i=1}^K \sum_{\mathrm{p} \in P}\omega_{\mathrm{p}}f(\varphi_i(\mathrm{p}))
\sqrt{\det((DQ_{d,k}\varphi_i(p))^T DQ_{d,k}\varphi_i(p))}\,, \label{eq:HOSQ2}
\end{align}
where $\mathrm{p}\in P$, $\omega_\mathrm{p} \in \R^+$ can be the points and weights of any quadrature of $\square_2$  e.g. the Gauss-Legendre or Clenshaw-Curtis quadrature  \cite{Xiang2012OnTC,Trefethen2019}. 

It is important to note that computing nodes and weights for a $n$-point Gauss-Legendre rule requires $\mathcal{O}(n^2)$ operations (error analysis is given in \cite{zavalani2023highorder}), while the Clenshaw-Curtis method uses $\mathcal{O}(n\log n)$ operations with the \emph{Discrete Cosine Transform} for evaluation. We estimate the error of the latter: 

%

\begin{theorem}\label{C_C}
Consider a $C^r$ surface $S$, where $r \geq 1$, with $r$-regular quadrilateral re-parametrization $\varphi_i : \square_2 \to S$,
\[
\varphi_i = \varrho_i \circ \sigma = \pi_i \circ \tau_i\circ \sigma, \quad i=1,\dots,K.
\]
Let $\mathrm{p}\in P$, and $\omega_\mathrm{p}$ be the points and weights of the tensorial $(n+1)$-order Clenshaw-Curtis quadrature rule,  $f : S \to \R$ be a function with absolutely continuous derivatives up to order $(r-1)$ and bounded variation $\|f^{(r)}\|_T = V < +\infty$, such that $f$ induces a negligible "remainder of the remainder" \cite{elliott2011estimates}. Then the integration error can be estimated as 
\begin{equation}\label{main_cc}
\mathcal{E}_{f} = \sum_{i=1}^K\int_{\square_2} f(\varphi(\mathrm{x}))g_i\left(\mathrm{x}\right)d\mathrm{x}
- \sum_{i=1}^K \sum_{\mathrm{p} \in P}\omega_{\mathrm{p}}f(\varphi_i(\mathrm{p})) g_i\left(\mathrm{p}\right) \leq \frac{128 \tilde V}{15\pi r(n-r)^{r}}, 
\end{equation}
where $g_i\left(\mathrm{x}\right)=\sqrt{\det((DQ_{d,k}\varphi_i(\mathrm{x}))^T DQ_{d,k}\varphi_i(\mathrm{x}))}$ and $\tilde V$ is defined as
\begin{equation}\label{eq.var}
\tilde V = \max_{i=1,\ldots,K}\max_{\alpha,\beta\leq r}\left\| \frac{\partial^{\alpha+\beta}}{\partial x^\alpha \partial y^\beta}\Big(f(\varphi_i(\mathrm{x}))g_{i}(\mathrm{x})\Big) \right\|_T\; \text{with}\; \|\xi\|_{T} = \left\|\frac{\xi'}{\sqrt{1-x^2}}\right\|_{1}.
\end{equation}
\end{theorem}

\begin{proof}
We denote $\psi(x,y):=f(\varphi(\mathrm{x}))g_i\left(\mathrm{x}\right)$ and 
\begin{equation}
R^{(2)}_n [ \psi ] = I^{(2)}[ \psi ] - Q^{(2)}_n [ \psi ] 
\end{equation}
as the remainder of the exact integral and the $n+1$-order quadrature rule
\begin{equation}
I^{(2)}[ \psi ] = \int_{\square_2} \psi(x, y) \, dx \, dy, \quad Q^{(2)}_n [ \psi ] = \sum_{j=0}^{n} \sum_{i=0}^{n} \psi(x_i, y_j) \, w_i \, w_j  \,.
\end{equation}
For a bivariate function $\psi(x, y)$, $I_x[\psi] = \int_{-1}^{1} \psi(x, \cdot) \, dx$ denotes the integration with respect to the $x$ variable only, yielding a function of $y$. The subscript notation extends to $R_{n,x}[\psi]$ and $Q_{n,y}[\psi]$, and when replacing the roles of $x$ and $y$. Fubini's theorem \cite{Brezis2010FunctionalAS} implies:
\begin{equation}\label{int_1}
I^{(2)}[\psi] = \int_{\square_2} \psi(x, y) \, dx \, dy = \int_{-1}^{1} I_x[\psi] \, dy = I_y[I_x[\psi]]\,.
\end{equation}
Upon substitution into Eq.~\eqref{int_1}, we obtain
\begin{align}\label{eq.20}
I^{(2)}[\psi] &= \nonumber
I_y[R_{n,x}[\psi] + Q_{n,x}[\psi]] \\ \nonumber
&=
R_{n,y}[R_{n,x}[\psi] + Q_{n,x}[\psi]] + Q_{n,y}[R_{n,x}[\psi] + Q_{n,x}[\psi]] \nonumber\\
&= 
R_{n,y}[R_{n,x}[\psi]] + Q_{n,x}[R_{n,y}[\psi]] + Q_{n,y}[R_{n,x}[\psi]] + Q_{n,y}[Q_{n,x}[\psi]] 
\end{align}
Following \cite{elliott2011estimates} we assume
the “remainder of a remainder” -- first term in Eq.~\eqref{eq.20} to contribute negligibly to the error, enabling us to establish a sufficiently tight upper bound:

For large $n$, the quadrature rule approaches the value of the integral, i.e., $Q_{n,\beta} \approx I_\beta$ for $\beta = x, y$, we are left with:
\begin{equation}
     I^{(2)}[\psi] \approx I_x[R_{n,y}[\psi]] + I_y[R_{n,x}[\psi]] + Q^{(2)}_n[\psi] \,.
\end{equation}
Hence, 
\begin{equation}\label{int_2}
R^{(2)}_n[\psi] \leq I_x[R_{n,y}[\psi]] + I_y[R_{n,x}[\psi]],
\end {equation}

As noted in \cite{Trefethen2019}, considering a  function $\theta(x)$ defined on the interval $[-1, 1]$, when computing $Q_n[\theta]$ using Clenshaw-Curtis quadrature for $\theta\in C^r$ and $\|\theta^{(r)}\|_{T}<V$ for a real finite value $V$, then for sufficiently large $n$, the subsequent inequality  holds:
\begin{equation}\label{res_tren}
    R_n[\theta]\leq \frac{32 V}{15\pi r(n-r)^{r}}
\end{equation}

Consequently, applying Eq.~\eqref{res_tren} to Eq.~\eqref{int_2} yields
\begin{equation}
R^{(2)}_n[\psi] \leq \frac{32}{15\pi r(n - r)^{r}}[ I_x[V_y(x)] + I_y[V_x(y)]]\,,
\end{equation} 
where \(V_y(x)= \max_y \|\psi^{(r)}(x,y)\|_T \) and \(V_x(y)= \max_x \|\psi^{(r)}(x,y)\|_T \) for fixed \( x \) and \( y \), respectively.

By utilizing the definition of the variation \( \tilde{V} \) in Eq.~\eqref{eq.var},
we are left with Eq.~\eqref{main_cc}.  Applying  Clenshaw-Curtis quadrature leads to an error approximation for a 2D surface integral with an order of $\Oc\big(n^{-r}\big)$, where $n$ denotes the quadrature's order.
\end{proof}

\section{Numerical experiments}\label{sec:NUM}
We demonstrate the effectiveness of $\text{HOSQ-CC}$ (utilizing Clenshaw-Curtis quadrature) for surface integration through two set of numerical examples.  We provide the initial flat surface triangulations by using Persson and Strang's algorithm~\cite{Persson} and subsequently enhance them to curved triangulations using Euclidean closest-point projections $\pi_i : T_i \to S$.
Our implementation of HOSQ is part of a {\sc Python} package {\sc surfpy} 
\footnote{available  at \url{https://github.com/casus/surfpy}
}.
The examples and results of this manuscript using \textsc{Dune-CurvedGrid}  
are summarized and made available in the repository.%
\footnote{\url{https://github.com/casus/dune-surface_int}
}
%

We compare $\text{HOSQ-CC}$ with \textsc{Dune-CurvedGrid} integration algorithm (DCG) from the surface-parametrization module \texttt{dune-curvedgrid} \cite{CurvedGrid}, part of  the \textsc{Dune} finite element framework\footnote{\url{www.dune-project.org}}. As outlined in \cite{zavalani2024note}, DCG performs total degree interpolation in uniform (equidistant midpoint-rule) triangle-nodes for interpolating the closest-point projection on each triangle, which makes it sensitive to Runge's phenomenon (overfitting).

\begin{figure}[t!]

\begin{subfigure}{.45\textwidth}
  \centering
  \includegraphics[clip,width=0.9\columnwidth]{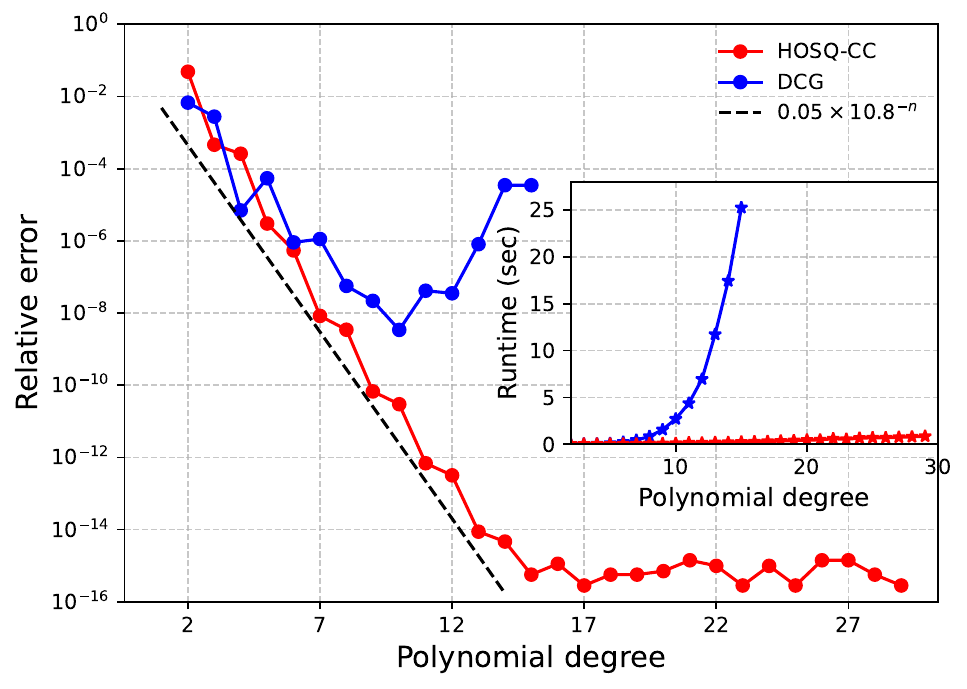}
  \caption{Unit sphere}
  \label{fig:US3}
\end{subfigure}%
\hfill
\begin{subfigure}{.45\textwidth}
  \centering
 \includegraphics[clip,width=0.9\columnwidth]{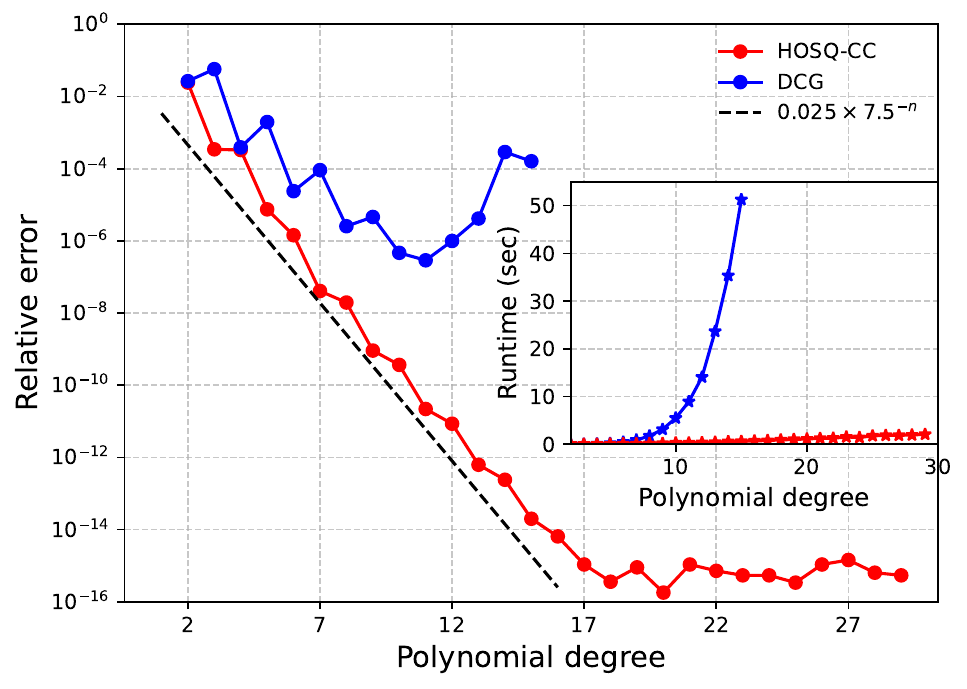}
  \caption{Torus with radii $R=2$ and $r=1$}
  \label{fig:US4}
\end{subfigure}
\caption{Relative errors and runtimes between DCG and $\text{HOSQ-CC}$ computing the surface area of the unit sphere (\ref{fig:US3}) and torus (\ref{fig:US4}) respectively.}
\label{fig:R}
\end{figure}
We start, by addressing integration tasks involving  only constant integrands $f : S \to \R$. 

\begin{experiment}[Area computation]\label{exp:1}
We integrate $f=1$ over the unit sphere $\mathrm{S}^2$ and the torus $T^2_{r,R}$ with inner radius $r=1$ and outer radius $R=2$. The surface areas are given due to $4\pi r^2$ and $4\pi^2 Rr$, respectively, enabling measuring the relative errors.
\end{experiment}

We choose initial triangulations of size $N_{\Delta}=124$  for the sphere and of size $N_{\Delta}=256$ for the torus and apply Clenshaw-Curtis quadrature\cite{Trefethen2019}, with a degree matching the geometry approximation.  We use symmetric Gauss quadrature on a simplex with a matching degree for the geometry approximation in the case of DCG. 

$\text{HOSQ-CC}$ exhibits stable convergence to machine precision with exponential rates, $0.05\times 10.8^{-n}$ and $0.025\times 7.5^{-n}$ fitted for the sphere and torus, respectively,  in accordance with the predictions from Theorem \ref{C_C}. In contrast, DCG becomes unstable for orders larger than $\deg =10$. We interpret the instability as the appearance of Runge's phenomenon caused by the choice of midpoint-triangle-refinements yielding equidistant interpolation nodes for DCG. 

In terms of execution time, the {\sc Python} implementation of $\text{HOSQ-CC}$, empowered by spectral differentiation, outperforms the $\text{C}++$ implementation of DCG.  Importantly, experimental findings reveal that DCG encounters constraints for $\deg \geq 15$, attributable to the formation of a singular matrix during the basis computation.

\begin{experiment}[Gauss-Bonet validation]
We consider the Gauss curvature as a non-trivial
integrand. 
Due to the Gauss--Bonnet theorem \cite{spivak1999}, integrating the Gauss curvature over a closed surface yields
\begin{equation}\label{eq:GB}
    \int_{S}K_{\mathrm{Gauss}}\,dS=2\pi \chi\left(S\right)\,,
\end{equation}
where $\chi\left(S\right)$ denotes the Euler characteristic of the surface. Here, we consider: \\

\noindent
\begin{tabular}{lll}
1) &  Dziuk's surface&  $(x-z^2)^2+y^2+z^2-1 = 0$\\
2) & Double torus & $\big(x^2+y^2)^2-x^2+y^2\big)^2+z^2-a^2=0$ ,\quad  $a \in \R\setminus\{0\}$\\
\end{tabular}

The Gauss curvature is computed symbolically from the
implicit surface descriptions using \textsc{Mathematica~11.3}, enabling measuring errors of DCG and $\text{HOSQ-CC}$ when integrating the Gauss curvature.
We maintain the experimental design outlined in { \sc Experiment}~\ref{exp:1} and display error plots based on the polynomial degree in Fig.\ref{fig:SURF1} and Fig.\ref{fig:SURF2}.
\end{experiment}

\begin{figure}[!htbp]
\begin{subfigure}{.45\textwidth}
\vfill
\centering
\vspace{-0.5cm}
\includegraphics[scale=0.4]{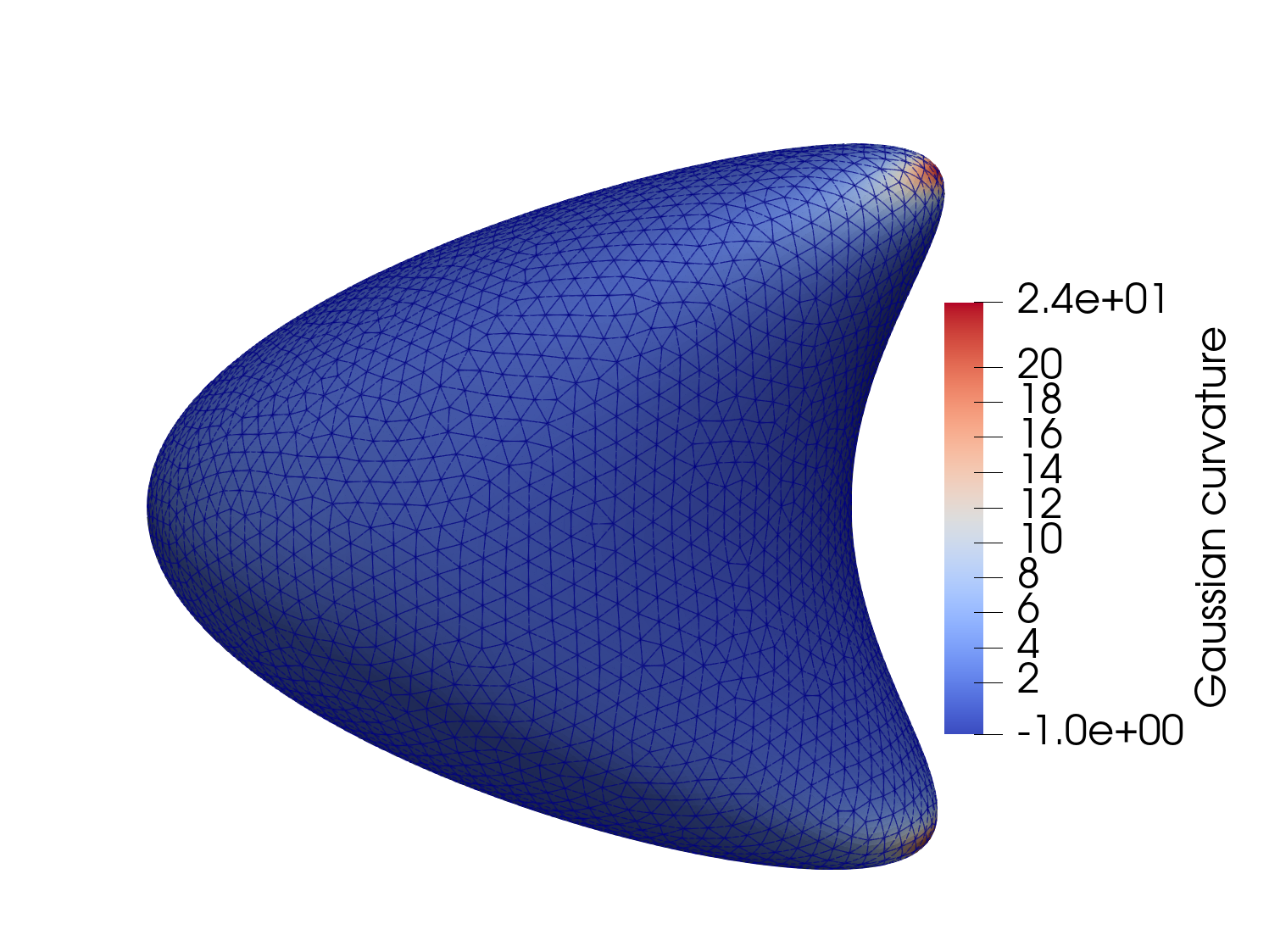}
\caption{Dziuk's surface with $8088$ triangles}
\label{gc.e}
\end{subfigure}
\hfill
\begin{subfigure}{.495\textwidth}
 \includegraphics[width=0.9\textwidth]{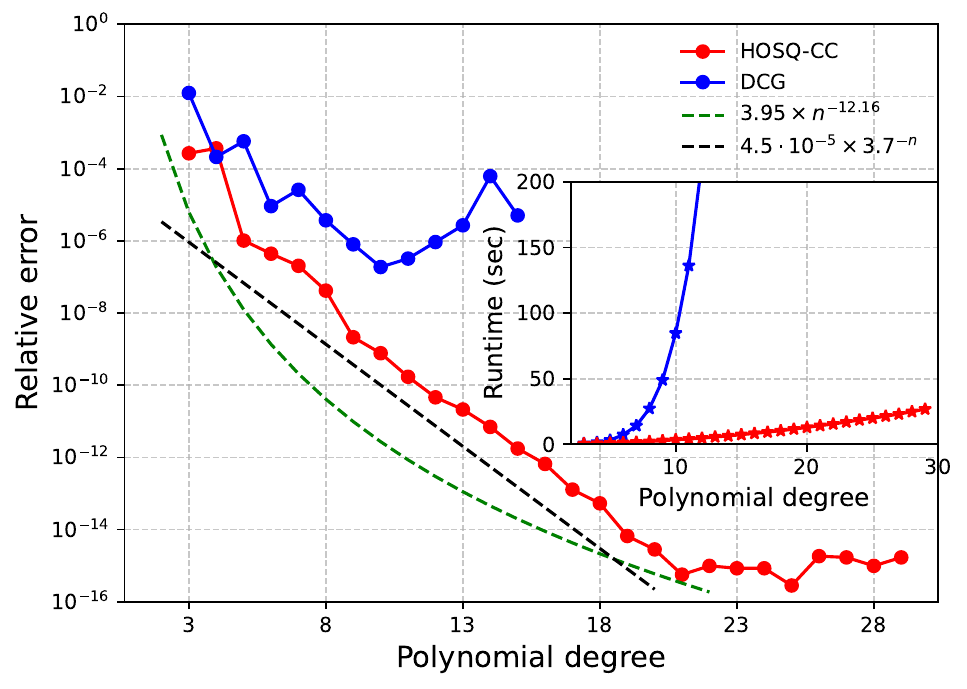}
  \caption{Integration errors and and runtimes}
  \label{r1.f}
\end{subfigure}%
\caption{Gauss-Bonnet validation for Dziuk's surface.}
\label{fig:SURF1}
\end{figure}
\begin{figure}[!htbp]
\begin{subfigure}{.45\textwidth}
\centering
\includegraphics[scale=0.35]{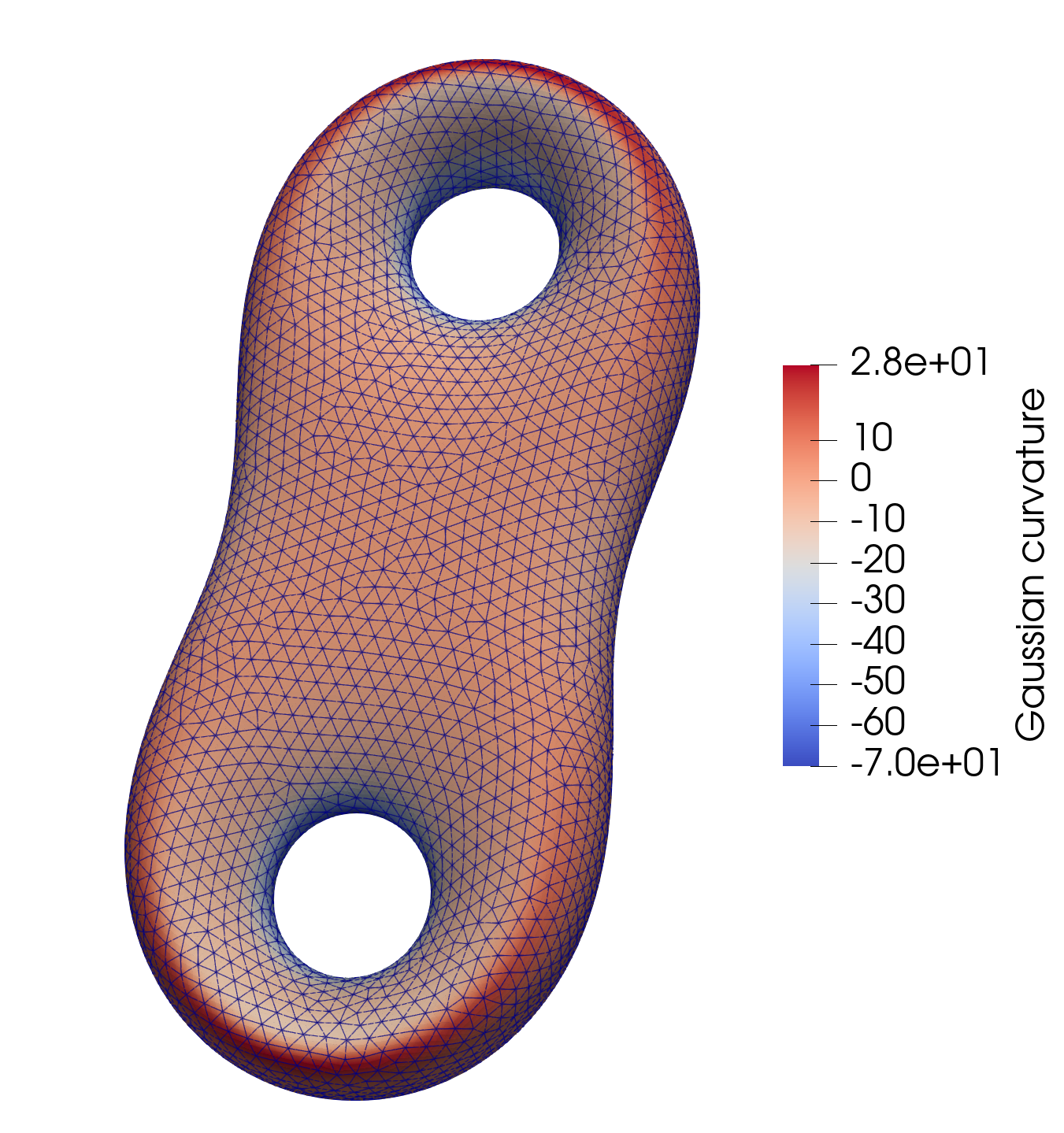}
\caption{Double torus, with $8360$ triangles and $a=0.2$}
\label{gc.d}
\end{subfigure}
\hfill
\begin{subfigure}{.495\textwidth}
 \includegraphics[width=0.9\textwidth]{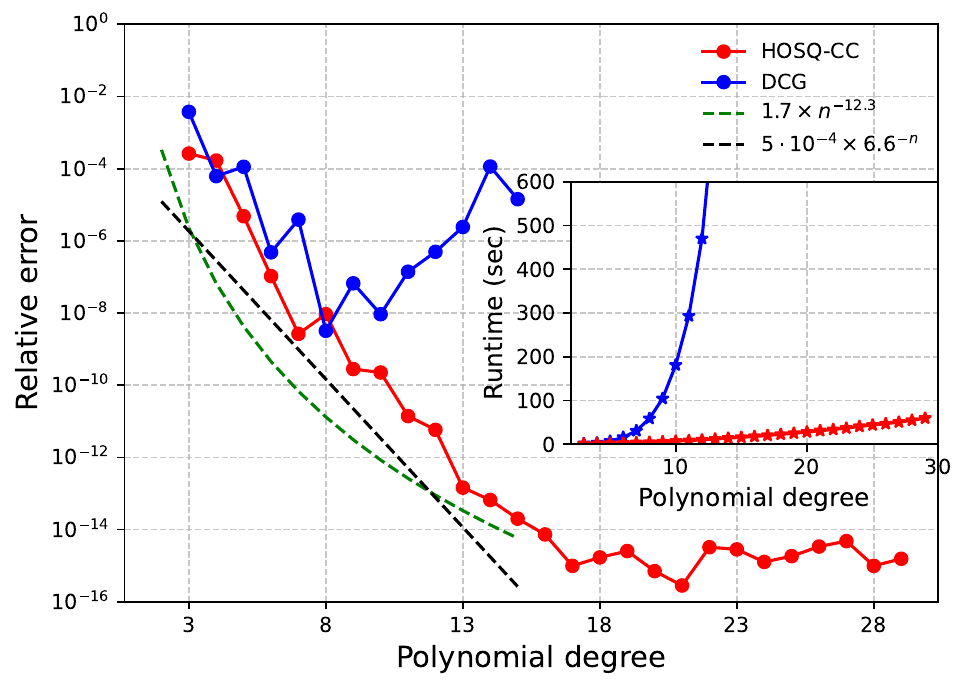}
  \caption{Integration errors and and runtimes}
  \label{r1.e}
\end{subfigure}
\caption{Gauss-Bonnet validation for a double torus.}
\label{fig:SURF2}
\end{figure}

The $\text{HOSQ-CC}$ consistently exhibits exponential convergence to the accurate value $2\pi\chi(S)$ with exponential rates, $4.5\cdot10^{-5}\times 3.7^{-n}$ and $5\cdot10^{-4}\times 6.6^{-n}$ fitted for the for the Dziuk's surface and the double torus, respectively. The best fit of an algebraic rate, $3.95\times n^{-12.16}$ for the Dziuk's surface and $1.7\times n^{-12.3}$ for the double torus, does not assert rapid convergence. As the order increases, the error in $\text{HOSQ-CC}$ tends to stabilize near the level of machine precision. In contrast, DCG fails to achieve machine-precision approximations across all cases and becomes unstable for degrees $k \geq 10$. In terms of execution time, once more, $\text{HOSQ-CC}$, empowered by spectral differentiation, outperforms DCG.
\section{Conclusion}
The present HOSQ integration approach, utilizing the innovative \emph{square-squeezing} transformation excels in speed, accuracy, and robustness for integration task on complex geometries, suggesting its application potential for triangular spectral element methods (TSEM) \cite{karniadakis2005spectral,heinrichs2001spectral} and fast spectral PDE solvers on surfaces \cite{fortunato2022high}.


\section*{Acknowledgement}
This research
was partially funded by the Center of Advanced Systems Understanding (CASUS), which is financed by Germany’s Federal Ministry of Education and Research (BMBF) and by the Saxon Ministry for Science, Culture, and Tourism (SMWK) with tax funds on the basis of the budget approved by the Saxon State Parliament.

\bibliographystyle{abbrv}
\bibliography{references}

\end{document}